\documentclass[11pt]{article}

\usepackage{amsthm}
\usepackage{latexsym}
\usepackage{dsfont}
\usepackage{bbm}
\usepackage{amssymb}
\usepackage{amsmath}
\usepackage{graphicx}
\usepackage{color}
\usepackage{wasysym}
\usepackage[noblocks,affil-it]{authblk}

\numberwithin{equation}{section}

\theoremstyle{plain}
\newtheorem{Theorem}{Theorem}[section]
\newtheorem{Lemma}[Theorem]{Lemma}
\newtheorem{Cor}[Theorem]{Corollary}
\newtheorem{Prop}[Theorem]{Proposition}
\theoremstyle{remark}
\newtheorem{Rem}[Theorem]{Remark}
\theoremstyle{definition}

\DeclareMathOperator{\N}{\mathbb{N}}

\DeclareMathOperator{\R}{\mathbb{R}}

\newcommand{\mn}{\mathbb{N}}
\newcommand{\me}{\mathbb{E}}

\DeclareMathOperator{\Prob}{\mathbb{P}}

\DeclareMathOperator{\E}{\mathbb{E}}

\DeclareMathOperator{\1}{\mathbbm{1}}

\title{A law of the iterated logarithm for the number of occupied boxes in the Bernoulli sieve}
\date{\today}

\author{Alexander Iksanov\footnote{Faculty of Computer Science and Cybernetics, National Taras
Shevchenko University of Kyiv, 01601 Kyiv, Ukraine, e-mail:
iksan@univ.kiev.ua}, Wissem Jedidi \footnote{Department of
Statistics  \& OR, King Saud University, P.O. Box 2455, Riyadh
11451, Saudi Arabia and Universit\'e de Tunis El Manar, Facult\'e
des Sciences de Tunis, LR11ES11 Laboratoire d'Analyse
Math\'ematiques et Applications, 2092, Tunis, Tunisia, e-mail:
wissem$_-$jedidi@yahoo.fr} \ and \ Fethi Bouzeffour
\footnote{Department of Mathematics, College of Sciences, King
Saud University, Riyadh 11451, Saudi Arabia, e-mail:
fbouzaffour@ksu.edu.sa}}
\begin{document}
\maketitle
\begin{abstract}
The Bernoulli sieve is an infinite occupancy scheme obtained by
allocating the points of a uniform $[0,1]$ sample over an infinite
collection of intervals made up by successive positions of a
multiplicative random walk independent of the uniform sample. We
prove a law of the iterated logarithm for the number of non-empty
(occupied) intervals as the size of the uniform sample becomes
large.
\end{abstract}

\bigskip



\section{Introduction}

Let $R:=(R_k)_{k\in\mn_0}$ be a multiplicative random walk defined
by
$$
R_0:=1, \quad R_k:=\prod_{i=1}^k W_i, \quad k\in\mn
$$
where $(W_k)_{k\in\mn}$ are independent copies of a random
variable $W$ taking values in the open interval $(0,1)$. Also, let
$(U_j)_{j\in\mn}$ be independent random variables which are
independent of $R$ and have the uniform distribution on $[0,1]$. A
random occupancy scheme in which `balls'\, $U_1$, $U_2$, etc.\ are
allocated over an infinite array of `boxes'\, $(R_k, R_{k-1}]$,
$k\in\mn$ is called {\it Bernoulli sieve}. The Bernoulli sieve was
introduced in \cite{Gnedin:2004} and further investigated in
numerous articles which can be traced via the references given in
the recent work \cite{Alsmeyer+Iksanov+Marynych:2016}. We also
refer to \cite{Alsmeyer+Iksanov+Marynych:2016} for more details
concerning the Bernoulli sieve including the origin of this term.

Since a particular ball falls into the box $(R_k, R_{k-1}]$ with
random probability
\begin{equation}\label{frequencies}
p^\ast_k:=R_{k-1}-R_k=W_1W_2\cdot\ldots\cdot W_{k-1}(1-W_k),
\end{equation}
the Bernoulli sieve is also the classical infinite occupancy
scheme with the {\it random probabilities} $(p^\ast_k)_{k\in\mn}$.
In this setting, given the random probabilities $(p^\ast_k)$, the
balls are allocated over the boxes $(R_1,R_0], (R_2,R_1],\ldots$
independently with probability $p^\ast_j$ of hitting box $j$.
Assuming that the number of balls equals $n$, denote by $K^\ast_n$
the number of non-empty boxes.

Under the condition $\sigma^2:={\rm Var}|\log W|\in (0,\infty)$
(which implies that $\mu:=\me |\log W|<\infty$) it was shown in
Corollary 1.1 of \cite{Gnedin+Iksanov+Marynych:2010} that, as
$n\to\infty$,
$$\frac{K^\ast_{[e^n]}-\mu^{-1}\int_0^n\Prob\{|\log(1-W)|\leq
y\}{\rm d}y}{\sqrt{\sigma^2\mu^{-3}n}}$$ converges in distribution
to the standard normal law. The same conclusion can also be
derived from a functional limit theorem obtained recently in
\cite{Alsmeyer+Iksanov+Marynych:2016}. The purpose of the present
article is to obtain a law of the iterated logarithm that
corresponds to the aforementioned central limit theorem.

For a family or a sequence $(x_t)$ denote by $C((x_t))$ the set of its limit points.
\begin{Theorem}\label{main}
Assume that $\sigma^2\in (0,\infty)$ and that $\me
|\log(1-W)|^a<\infty$ for some $a>0$. Then
$$C\bigg(\bigg(\frac{K^\ast_{[e^n]}-\mu^{-1}\int_0^n\Prob\{|\log(1-W)|\leq y\}{\rm d}y}{\sqrt{2\sigma^2\mu^{-3}n\log\log
n}}:n\geq 3\bigg)\bigg)=[-1,1]\quad\text{{\rm a.s.}}$$ In
particular,
$${\lim\sup\,(\lim
\inf)}_{n\to\infty}\frac{K^\ast_{[e^n]}-\mu^{-1}\int_0^n\Prob\{|\log(1-W)|\leq
y\}{\rm d}y}{\sqrt{2n\log\log
n}}=+(-)\sigma\mu^{-3/2}\quad\text{{\rm a.s.}}$$ % where $\mu:=\me
%|\log W|<\infty$.
\end{Theorem}

The proof of Theorem \ref{main} given in Section \ref{pr} relies
upon a number of auxiliary results that are stated and proved in
Section \ref{au}.

\section{Proof of Theorem \ref{main}}\label{pr}

Let $(\xi_k, \eta_k)_{k\in\mn}$ be a sequence of i.i.d.\
two-dimensional random vectors with generic copy $(\xi,\eta)$
where both $\xi$ and $\eta$ are positive. No
condition is imposed on the dependence structure between $\xi$ and
$\eta$. Set
\begin{equation*}
N(x)~:=~\sum_{k\geq 0}\1_{\{S_k+\eta_{k+1}\leq x\}}, \quad x \geq
0
\end{equation*}
where $(S_n)_{n\in\mn_0}$ is the zero-delayed ordinary random walk
with increments $\xi_n$ for $n\in\mn$, i.e., $S_0 = 0$ and $S_n =
\xi_1+\ldots+\xi_n$, $n \in \mn$.

Set $\rho^\ast(x):=\sum_{k\geq 0}\1_{\{p_k^\ast\geq 1/x\}}$ for
$x>0$. It is natural to call $\rho^\ast(x)$ the number of `large
boxes' in the Bernoulli sieve. Relevance of $N(x)$ to the present
context is justified by the equality
\begin{equation}\label{equ}
\rho^\ast(x)=N^\ast(\log x)
\end{equation}
where the random variable $N^\ast(x)$ corresponds to $\xi_k=|\log
W_k|$ and $\eta_k=|\log (1-W_k)|$ for $k\in\mn$.

Our strategy is as follows. First, we show in Corollary
\ref{important2} that the number of occupied boxes $K_n^\ast$ is
well-approximated in the a.s.\ sense by $\rho^\ast(n)$. A similar
approximation in the sense of distributional convergence was
established in \cite{Gnedin+Iksanov+Marynych:2010} and
\cite{Alsmeyer+Iksanov+Marynych:2016}. We would like to stress
that proving the a.s.\ approximation is more delicate and calls
for an additional argument. Second, we prove in Proposition
\ref{imp2} a law of the iterated logarithm for $N(x)$ defined in
terms of arbitrary perturbed random walk. In view of \eqref{equ}
these two results are sufficient to complete the proof of Theorem
\ref{main}.

We consider the infinite occupancy scheme in which balls are
allocated independently with probability $p_k$ of hitting box $k$.
Denote by $K_n$ be the number of occupied boxes in the scheme when
$n$ balls have been thrown. For $n\in\mn$ set
$\Theta_n:=\sum_{k\geq 1}e^{-np_k}\1_{\{np_k\geq 1\}}$ and
$\Delta_n:=n \sum_{k\geq 1}p_k\1_{\{np_k<1\}}$.

\begin{Lemma}\label{aux-1}
Suppose that
\begin{equation}\label{cond1}
\sum_{n\geq 1}n^{-2}\Theta^4_{[e^n]}<\infty\quad\text{and}\quad
\sum_{n\geq 1}n^{-2}\Delta^4_{[e^n]}<\infty.
\end{equation}
Then
\begin{equation}\label{inter} \lim_{n\to\infty}\,
n^{-1/2}\big(K_{[e^n]}-\sum_{k\geq 1}\1_{\{[e^n]p_k\geq
1\}}\big)=0\quad\text{{\rm a.s.}}
\end{equation}
\end{Lemma}
\begin{proof}
We shall use a representation $K_n=\sum_{k\geq 1}\1_{\{Z_{n,k}\geq
1\}}$ where $Z_{n,k}$ is the number of balls that fall in box $k$.
Observe that the random variable $Z_{n,k}$ has the binomial
distribution with parameters $n$ and $p_k$. With this at hand we
can write $$\big|K_n-\sum_{k\geq 1}\1_{\{np_k\geq 1\}}\big|\leq
\sum_{k\geq 1}\1_{\{Z_{n,k}=0\}}\1_{\{np_k\geq 1\}}+\sum_{k\geq
1}\1_{\{Z_{n,k}\geq 1\}}\1_{\{np_k<1\}}.$$ Let $(A_k)_{k\in\mn}$
be a sequence of sets which satisfy $\sum_{k\geq
1}\1_{A_k}<\infty$. The multinomial theorem tells us that
\begin{eqnarray}\label{formula}
(\sum_{k\geq 1}\1_{A_k})^4&=&\sum_{k\geq 1}\1_{A_k}+14\sum_{1\leq
j<i}\1_{A_j}\1_{A_i}\notag\\&+& 36\sum_{1\leq
j<i<l}\1_{A_j}\1_{A_i}\1_{A_l}+24\sum_{1\leq
j<i<l<m}\1_{A_j}\1_{A_i}\1_{A_l}\1_{A_m}.
\end{eqnarray}
Further, for $m\in\mn$ and distinct $i_1,\ldots, i_m\in \mn$
\begin{equation}\label{form2}
\Prob\{Z_{n,\,i_1}=0,\ldots,
Z_{n,\,i_m}=0\}=(1-p_{i_1}-\ldots-p_{i_m})^n\leq
e^{-n(p_{i_1}+\ldots+p_{i_m})}.
\end{equation}
Even though there is a precise formula
\begin{eqnarray*}
\Prob\{Z_{n,\,i_1}\geq 1,\ldots, Z_{n,\,i_m}\geq
1\}&=&1-\sum_{k=1}^m(1-p_{i_k})^n+\sum_{1\leq k<l\leq m}
(1-p_{i_k}-p_{i_l})^n\\&+&\ldots+(-1)^m(1-p_{i_1}-\ldots-p_{i_m})^n,
\end{eqnarray*}
a crude upper bound is of greater use for our needs:
$$\Prob\{Z_{n,\,i_1}\geq 1,\ldots, Z_{n,\,i_m}\geq 1\}\leq n(n-1)\cdot\ldots\cdot (n-m+1)p_{i_1}\cdot\ldots\cdot p_{i_m}\leq n^m p_{i_1}\cdot\ldots\cdot p_{i_m}.$$ While the product $p_{i_1}\cdot\ldots\cdot p_{i_m}$ is the probability of the event that the boxes $i_1,\ldots, i_m$ turn out occupied when throwing $m$ balls, the product $n(n-1)\cdot\ldots\cdot (n-m+1)=m!\binom{n}{m}$ is the number of ways to allocate $m$ balls out of $n$ into the boxes $i_1,\ldots,i_m$.

Using \eqref{formula} and then \eqref{form2} we obtain
\begin{eqnarray*}
&&\me \bigg(\sum_{k\geq 1}\1_{\{Z_{n,k}=0\}}\1_{\{n p_k\geq
1\}}\bigg)^4\\&\leq& \sum_{k\geq 1} e^{-np_k}\1_{\{np_k\geq
1\}}+14\sum_{1\leq j<i}e^{-np_j}\1_{\{np_j\geq
1\}}e^{-np_i}\1_{\{np_i\geq 1\}}\\&+&36\sum_{1\leq
j<i<l}e^{-np_j}\1_{\{np_j\geq 1\}}e^{-np_i}\1_{\{np_i\geq
1\}}e^{-np_l}\1_{\{np_l\geq 1\}}\\&+&24\sum_{1\leq
j<i<l<m}e^{-np_j}\1_{\{np_j\geq 1\}}e^{-np_i}\1_{\{np_i\geq
1\}}e^{-np_l}\1_{\{np_l\geq 1\}}e^{-np_m}\1_{\{np_m\geq
1\}}\\&\leq& \Theta_n+7\Theta_n^2+6\Theta_n^3+\Theta_n^4.
\end{eqnarray*}
This in combination with \eqref{cond1} entails
$\lim_{n\to\infty}\, n^{-1/2}\sum_{k\geq
1}\1_{\{Z_{[e^n],k}=0\}}\1_{\{[e^n]p_k\geq 1\}}=0$ a.s.\ by the
Borel-Cantelli lemma.

Arguing similarly we infer
\begin{eqnarray*}
&&\me \bigg(\sum_{k\geq 1}\1_{\{Z_{n,k}\geq 1\}}\1_{\{n p_k<
1\}}\bigg)^4\\&\leq& n\sum_{k\geq 1}
p_k\1_{\{np_k<1\}}+14n^2\sum_{1\leq
j<i}p_j\1_{\{np_j<1\}}p_i\1_{\{np_i<1\}}\\&+&36n^3\sum_{1\leq
j<i<l}p_j\1_{\{np_j<1\}}p_i\1_{\{np_i<1\}}p_l\1_{\{np_l<
1\}}\\&+&24n^4\sum_{1\leq j<i<l<m}p_j\1_{\{np_j<1\}}p_i\1_{\{np_i<
1\}}p_l\1_{\{np_l<1\}}p_m\1_{\{np_m<1\}}\\&\leq&
\Delta_n+7\Delta_n^2+6\Delta_n^3+\Delta_n^4
\end{eqnarray*}
which in combination with \eqref{cond1} proves
$\lim_{n\to\infty}\, n^{-1/2}\sum_{k\geq 1}\1_{\{Z_{[e^n],k}\geq
1\}}\1_{\{[e^n]p_k< 1\}}=0$ a.s.\ by another appeal to the
Borel-Cantelli lemma.
\end{proof}

\begin{Cor}\label{important2}

$$\lim_{n\to\infty}\, n^{-1/2}\big(K^\ast_{[e^n]}-\rho^\ast(e^n)\big)~=~ 0\quad\text{{\rm a.s.}}$$

\end{Cor}
\begin{proof}
Recalling \eqref{equ} we have
$$0\leq \rho^\ast(e^n)-\rho^\ast([e^n])=N^\ast(n)-N(\log [e^n])\leq N^\ast(n)-N^\ast(n-1)$$ for large
enough $n$. By Lemma \ref{aux2}(b), the right-hand side divided by
$n^{1/2}$ converges to zero a.s. Hence, it suffices to prove that
\begin{equation}\label{imp}
\lim_{n\to\infty}\,
n^{-1/2}\big(K^\ast_{[e^n]}-\rho^\ast([e^n])\big)~=~
0\quad\text{a.s.}
\end{equation}
We have
\begin{eqnarray*}
\Delta_n^\ast&:=&n\sum_{k\geq 1}p_k^\ast\1_{\{n
p_k^\ast<1\}}=n\int_{(n,\,\infty)}x^{-1}{\rm
d}\rho^\ast(x)=\int_{(1,\,\infty)}x^{-1}{\rm
d}(\rho^\ast(nx)-\rho^\ast(n))\\&=&\int_1^\infty
x^{-2}(\rho^\ast(nx)-\rho^\ast(n)){\rm d}x
\end{eqnarray*}
having utilized integration by parts and the asymptotics
$\rho^\ast(x)=O(\log x)$ as $x\to\infty$ a.s.\ (see Lemma
\ref{aux2}(a)) for the last step. Further, using convexity of
$x\mapsto x^4$, $x>0$ and Corollary \ref{impo3} yields
\begin{eqnarray*}
\me (\Delta^\ast_n)^4&=&\me \bigg(\sum_{k\geq 2}\int_{k-1}^k
x^{-2}(\rho^\ast(nx)-\rho^\ast(n)){\rm d}x\bigg)^4\\&\leq& \me
\bigg(\sum_{k\geq
2}(\rho^\ast(nk)-\rho^\ast(n))((k-1)k)^{-1}\bigg)^4\\&\leq&
\sum_{k\geq 2}\me (\rho^\ast(nk)-\rho^\ast(n))^4((k-1)k)^{-1}\leq
C \sum_{k\geq 2}(\log k)^4((k-1)k)^{-1}<\infty
\end{eqnarray*}
which proves $$\sum_{n\geq 1}n^{-2}\me
(\Delta^\ast_{[e^n]})^4<\infty.$$ Arguing similarly we obtain
$$\Theta_n^\ast:=\sum_{k\geq 1}e^{-n p_k^\ast}\1_{\{n p_k^\ast\geq
1\}}=\int_{[1,n]}e^{-n/x}{\rm d}\rho^\ast(x)=\int_1^n e^{-x}
(\rho^\ast(n)-\rho^\ast(n/x)){\rm d}x$$ and
\begin{eqnarray*}
&&\me (\Theta_n^\ast)^4\\&=&\me \bigg(\sum_{k=2}^n \int_{k-1}^k
e^{-x}(\rho^\ast(n)-\rho^\ast(n/x)){\rm d}x\bigg)^4\\&\leq&
(e^{-1}-e^{-n})^4 \me \bigg(\sum_{k\geq
2}(\rho^\ast(n)-\rho^\ast(n/k))(e^{-k+1}-e^{-k})(e^{-1}-e^{-n})^{-1}\bigg)^4\\&\leq&
(e-1)(e^{-1}-e^{-n})^3 \sum_{k\geq 2}\me
(\rho^\ast(n)-\rho^\ast(n/k))^4 e^{-k}\leq C \sum_{k\geq 2}(\log
k)^4e^{-k}<\infty.
\end{eqnarray*}
Thus, $$\sum_{n\geq 1}n^{-2}\me (\Theta^\ast_{[e^n]})^4<\infty.$$
Invoking now Lemma \ref{aux-1} enables us to conclude that
\eqref{imp} holds conditionally on $(p_k^\ast)_{k\in\mn}$, hence
also unconditionally. The proof of Corollary \ref{important2} is
complete.
\end{proof}

\begin{Prop}\label{imp2}
Suppose that ${\tt s}^2:={\rm Var}\,\xi\in (0,\infty)$ and $\me
\eta^a<\infty$ for some $a>0$. Then $$C\bigg(\bigg(\frac{N(n)-{\tt
m}^{-1}\int_0^n F(y){\rm d}y}{\sqrt{2{\tt s}^2{\tt
m}^{-3}n\log\log n}}:n\geq 3\bigg)\bigg)=[-1,1]\quad\text{{\rm
a.s.}},$$ where ${\tt m}:=\me \eta<\infty$ and
$F(y):=\Prob\{\eta\leq y\}$ for $y\geq 0$.
\end{Prop}
\begin{proof}
Put $$\nu(x):=\sum_{k\geq 0}\1_{\{S_k\leq x\}},\quad x\geq 0.$$ It
is known (see the proof of Theorem 3.2 in
\cite{Alsmeyer+Iksanov+Marynych:2016}) that
$$\lim_{n\to\infty}\,n^{-1/2}\bigg(N(n)-\int_{[0,\,n]}F(n-y){\rm d}\nu(y)\bigg)=0\quad\text{a.s.}$$ whenever $\me \eta^a<\infty$ for some $a>0$ (the finiteness of ${\rm Var}\,\xi$ is not needed). Thus, it remains to prove that
\begin{equation}\label{inter1}
C\bigg(\bigg(\frac{\int_{[0,\,n]}F(n-y){\rm d}(\nu(y)-{\tt m}^{-1}y)}{\sqrt{2{\tt s}^2{\tt m}^{-3}n\log\log n}}:n\geq 3\bigg)\bigg)=[-1,1]\quad\text{a.s.}
\end{equation}

Put $a(t):=\sqrt{2{\tt s}^2{\tt m}^{-3}t\log\log t}$ for $t\geq 3$. Integrating by parts yields
\begin{eqnarray*}
\frac{\int_{[0,\,n]} F(n-y){\rm d}(\nu(y)-{\tt m}^{-1}y)-\Prob\{\xi=n\}}{a(n)}&=&\int_{[0,\,n)}\frac{\nu(n-y)-{\tt m}^{-1}(n-y)}{a(n)}{\rm
d}F(y)\\&=&\int_{[0,\,\delta]} {\nu(n-y)-{\tt m}^{-1}(n-y)\over
a(n)}\,{\rm d}F(y)\\&+&\int_{(\delta,\,n)}{\nu(n-y)-{\tt m}^{-1}(n-y)\over
a(n)}\,{\rm d}F(y)\\&=:& Z_1(n)+Z_2(n)
\end{eqnarray*}
for any fixed $\delta\in (0,n]$. We have a.s.
\begin{eqnarray*}
{\nu(n)-{\tt m}^{-1}n \over
a(n)}F(\delta)&-&{\nu(n)-\nu(n-\delta)\over a(n)}F(\delta)\\&\leq&
Z_1(n)\\&\leq& {\nu(n)-{\tt m}^{-1}n\over
a(n)}F(\delta)+\frac{{\tt m}^{-1}\delta}{a(n)}F(\delta).
\end{eqnarray*}
Fix any $x_0\in [-1,1]$. According to \eqref{123}, there exists a
sequence $(n_k)$ satisfying $\lim_{k\to\infty}n_k=\infty$ a.s.\
and $\lim_{k\to\infty}(\nu(n_k)-{\tt m}^{-1}n_k)/a(n_k)=x_0$ a.s.
By Lemma \ref{aux2}(b),
$\lim_{k\to\infty}(\nu(n_k)-\nu(n_k-\delta))/ a(n_k)=0$ a.s.
Therefore, $\lim_{\delta\to\infty}\lim_{k\to\infty}\,
Z_1(n_k)=x_0$ a.s. Further,
\begin{eqnarray*}
&&-\frac{\sup_{y\in[0,\,n]}\,|\nu(y)-{\tt m}^{-1}y|}{a(n)}(F(n-)-F(\delta))\\&\leq&
\frac{\inf_{y\in[0,\,n-\delta]}\,(\nu(y)-{\tt m}^{-1}y)}{a(n)}(F(n-)-F(\delta))
\\&\leq& Z_2(n)\\&\leq& \frac{\sup_{y\in
[0,\,n]}\,|\nu(y)-{\tt m}^{-1}y|}{a(n)}(F(n-)-F(\delta)).
\end{eqnarray*}
Using \eqref{124} we conclude that
$${\lim}_{\delta\to\infty}{\lim\sup}_{k\to\infty}
Z_2(n_k)={\lim}_{\delta\to\infty}{\lim\inf}_{k\to\infty} Z_2(n_k)=
0\quad\text{a.s.}$$ The proof of \eqref{inter1} is complete.
\end{proof}

Now Theorem \ref{main} follows from Corollary \ref{important2} in
combination with a specialization of Proposition \ref{imp2} for
$\rho^\ast(e^n)=N^\ast(n)$ which reads
$$C\bigg(\bigg(\frac{\rho^\ast(e^n)-\mu^{-1}\int_0^n\Prob\{\log|1-W|\leq
y\}{\rm d}y}{\sqrt{2\sigma ^2 \mu^{-3}n\log\log n}} : n\geq
3\bigg)\bigg)=[-1,1]\quad\text{a.s.}$$

\section{Auxiliary results}\label{au}

The following result can be found in the proof of Lemma 7.3 in
\cite{Alsmeyer+Iksanov+Marynych:2016}.
\begin{Lemma}\label{impo1}
Let $G:[0,\infty)\to [0,\infty)$ be a locally bounded function.
Then, for any $l\in\N$
\begin{equation}\label{impo1-gen}
\E \bigg(\sum_{k\geq 0}G(t-S_k)\1_{\{S_k\leq t\}}\bigg)^l\leq
\bigg(\sum_{j=0}^{[t]}\sup_{y\in[j,\,j+1)}G(y)\bigg)^l\me
(\nu(1))^l,\quad t\geq 0.
\end{equation}
\end{Lemma}

\begin{Lemma}\label{impo4}
For $0\leq y<x$ with $x-y>1$ $\me (N(x)-N(y))^4\leq C (x-y)^4$ for
a positive constant $C$ which does not depend on $x$ and $y$.
\end{Lemma}
\begin{proof}
Throughout the proof we assume that $x$ and $y$ satisfy the
assumptions of the lemma.

We start with $$\me (N(x)-N(y))^4 \leq 8 (\me (X(x,y))^4+\me
(Y(x,y))^4),$$ where
\begin{eqnarray*}
X(x,y)&:=&\sum_{j\geq 0}\big((\1_{\{S_j+\eta_{j+1}\leq
x\}}-F(x-S_j)\1_{\{S_j\leq x\}})\\&-&(\1_{\{S_j+\eta_{j+1}\leq
y\}}-F(y-S_j)\1_{\{S_j\leq y\}})\big);
\end{eqnarray*}
\begin{eqnarray*}
Y(x,y)&:=&\sum_{j\geq 0}\big(F(x-S_j)\1_{\{S_j\leq
x\}}-F(y-S_j)\1_{\{S_j\leq
y\}}\big)\\&=&\int_{[0,y]}(\nu(x-z)-\nu(y-z)){\rm
d}F(z)+\int_{(y,x]}\nu(x-z){\rm d}F(z);
\end{eqnarray*}
$F(z)=\Prob\{\eta\leq z\}$ is the distribution function of $\eta$
and $\nu(z)=\sum_{k\geq 0}\1_{\{S_k\leq z\}}$ for $z\geq 0$.

We intend to show that $\me (X(x,y))^4\leq C(x-y)^2$. With
$x,y\geq 0$ fixed, $X(x,y)$ equals the terminal value of the
martingale $(R(k),\mathcal{F}_k)_{k\in\mn_0}$ where $R(0):=0$,
\begin{eqnarray*}
R(k)&:=&\sum_{j=0}^{k-1}\big((\1_{\{S_j+\eta_{j+1}\leq x\}}-F(x-S_j)\1_{\{S_j\leq x\}})\\&-&(\1_{\{S_j+\eta_{j+1}\leq y\}}
-F(y-S_j)\1_{\{S_j\leq y\}})\big),
\end{eqnarray*}
$\mathcal{F}_0:=\{\Omega, \oslash\}$ and
$\mathcal{F}_k:=\sigma((\xi_j,\eta_j):1\leq j\leq k)$. We use the
Burkholder-Davis-Gundy inequality (Theorem 11.3.2 in
\cite{Chow+Teicher:2003}) to obtain, for any $l\in\mn$
\begin{eqnarray*}
&&\E (X(x,y))^{2l}\notag\\
&\leq& C_l\bigg(\E\bigg(\sum_{k\geq 0}\E
\big((R(k+1)-R(k))^2|\mathcal{F}_k\big)\bigg)^l+\sum_{k\geq
0}\E
\big(R(k+1)-R(k)\big)^{2l}\bigg)\\
&=:&C_l(I_1+I_2)
\end{eqnarray*}
for a positive constant $C_l$. We shall show that
\begin{equation}\label{11}
I_1 \leq 2^l \me (\nu(1))^l(b(x-y))^l
\end{equation}
where $b(t):=\sum_{k=0}^{[t]+1}(1-F(k))$ for $t\geq 0$ and that
\begin{equation}\label{22}
I_2\leq 2^{2l} \me \nu(1) b(x-y).
\end{equation}
These estimates serve our needs because $b(t)\leq [t]+2\leq 3t$
whenever $t>1$.

\noindent {\sc Proof of \eqref{11}}. We first observe that
\begin{eqnarray*}
&&\sum_{k\geq 0}\E
\big((R(k+1)-R(k))^2|\mathcal{F}_k\big)\\&=&\int_{(y,\,x]}F(x-z)(1-F(x-z)){\rm
d}\nu(z)\\&+&\int_{[0,\,y]}(F(x-z)-F(y-z))(1-F(x-z)+F(y-z)){\rm
d}\nu(z)\\&\leq& \int_{(y,\,x]}(1-F(x-z)){\rm
d}\nu(z)+\int_{[0,\,y]}(F(x-z)-F(y-z)){\rm d}\nu(z)
\end{eqnarray*}
whence $$I_1 \leq 2^{l-1}\bigg(\me
\bigg(\int_{(y,\,x]}(1-F(x-z)){\rm d}\nu(z)\bigg)^l
+\me\bigg(\int_{[0,\,y]}(F(x-z)-F(y-z)){\rm
d}\nu(z)\bigg)^l\bigg)$$ having utilized $(u+v)^l\leq
2^{l-1}(u^l+v^l)$ for nonnegative $u$ and $v$. Using Lemma
\ref{impo1} with $G(z)=(1-F(z))\1_{[0,\, x-y)}(z)$ and
$G(z)=F(x-y+z)-F(z)$, respectively, we obtain
\begin{eqnarray}
&&\me \bigg(\int_{(y,\,x]}(1-F(x-z)){\rm
d}\nu(z)\bigg)^l\notag\\&=&\me
\bigg(\int_{[0,\,x]}(1-F(x-z))\1_{[0,\,x-y)}(x-z){\rm
d}\nu(z)\bigg)^l\notag\\&\leq& \me (\nu(1))^l
\bigg(\sum_{n=0}^{[x]}\sup_{y\in
[n,\,n+1)}((1-F(z))\1_{[0,\,x-y)}(z))\bigg)^l\notag\\&\leq& \me
(\nu(1))^l \bigg(\sum_{n=0}^{[x-y]}(1-F(n))\bigg)^l\leq \me
(\nu(1))^l (b(x-y))^l .\label{11a}
\end{eqnarray}
and
\begin{eqnarray}
&&\me\bigg(\int_{[0,\,y]}(F(x-z)-F(y-z)){\rm
d}\nu(z)\bigg)^l\notag\\&\leq& \me (\nu(1))^l
\bigg(\sum_{n=0}^{[y]}\sup_{z\in
[n,\,n+1)}(F(x-y+z)-F(z))\bigg)^l\notag\\&\leq& \me (\nu(1))^l
\bigg(\sum_{n=0}^{[y]}(1-F(n))-\sum_{n=0}^{[y]}(1-F(x-y+n+1))\bigg)^l \notag\\
&\leq& \me (\nu(1))^l\bigg(\sum_{n=0}^{[y]}(1-F(n))-\sum_{n=0}^{[y]+2}(1-F(n))+\sum_{n=0}^{[x-y]+1}(1-F(n))\bigg)^l\notag\\&\leq&
\me(\nu(1))^l (b(x-y))^l%\bigg(\sum_{n=0}^{[(u-v)t]+1}(1-F(n))\bigg)^l
.\label{11b}
\end{eqnarray}
Combining \eqref{11a} and \eqref{11b} for $l=2$ yields \eqref{11}.

\noindent {\sc Proof of \eqref{22}}. Let us calculate
\begin{eqnarray*}
&&\me ((R(k+1)-R(k))^4 |\mathcal{F}_k)\\&\leq
&8((1-F(x-S_k))^4 F(x-S_k)+(F(x-S_k))^4 (1-F(x-S_k)))\1_{\{y<S_k\leq
x\}}\\&+&((1-F(x-S_k)+F(y-S_k))^4 (F(x-S_k)-F(y-S_k))\\&+&(F(x-S_k)-F(y-S_k))^4 (1-F(x-S_k)+F(y-S_k)))\1_{\{S_k\leq
y\}}\\&\leq& 8((1-F(x-S_k))\1_{\{y<S_k\leq
x\}}+(F(x-S_k)-F(y-S_k))\1_{\{S_k\leq y\}}).
\end{eqnarray*}
Therefore, $$I_2\leq 8 \bigg(\me
\int_{(x,y]}(1-F(x-z)){\rm d}\nu(z)+\me
\int_{[0,\,y]}(F(x-z)-F(y-z)){\rm d}\nu(z)\bigg).$$ Using now
formulae \eqref{11a} and \eqref{11b} with $l=1$ yields \eqref{22}.

Passing to $Y(x,y)$ we have $$\me (Y(x,y))^4 \leq
8\bigg(\me\bigg(\int_{[0,\,y]}(\nu(x-z)-\nu(y-z)){\rm
d}F(z)\bigg)^4+\me (\nu(x-y))^4\bigg).$$ Using the fact that
$z^{-4}\me (\nu(z))^4$ converges as $z\to\infty$ to a nonnegative
constant (see Theorem 5.1 on p.~57 in \cite{Gut:2009}) we infer
$\me (\nu(x-y))^4\leq C(x-y)^4$ (recall that $x-y>1$). Finally,
\begin{eqnarray*}
&&\me\bigg(\int_{[0,\,y]}(\nu(x-z)-\nu(y-z)){\rm
d}F(z)\bigg)^4\\&\leq&
(F([y]+1))^4\me\bigg(\sum_{k=0}^{[y]}(\nu(x-k)-\nu(y-k-1))\frac{F(k+1)-F(k)}{F([y]+1)}\bigg)^4\\&\leq&
\sum_{k=0}^{[y]}\me (\nu(x-k)-\nu(y-k-1))^4 (F(k+1)-F(k))\\&\leq&
\me (\nu(x-y+1))^4\leq C(x-y)^4
\end{eqnarray*}
where we have used distributional subadditivity of $\nu(z)$ (see
formula (5.7) on p.~58 in \cite{Gut:2009}) for the penultimate
inequality.
\end{proof}

In view of \eqref{equ} the next result is an immediate consequence
of Lemma \ref{impo4}.
\begin{Cor}\label{impo3}
$\me (\rho^\ast(x)-\rho^\ast(y))^4\leq C (\log (x/y))^4$ for a
positive constant $C$ which does not depend on $x$ and $y$.
\end{Cor}

\begin{Lemma}\label{aux2}
(a) $N(x)=O(x)$ a.s.\ as $x\to\infty$;

\noindent (b) For any $c>0$ and any fixed $\delta>0$
$\lim_{n\to\infty}n^{-c}(N(n)-N(n-\delta))=0$ a.s. and
$\lim_{n\to\infty}n^{-c}(\nu(n)-\nu(n-\delta))=0$ a.s.
\end{Lemma}
\begin{proof}
(a) Since the $\eta_k$ is a.s.\ positive, it follows that
$N(x)\leq \nu(x)$ a.s. It remains to note that $\nu(x)=O(x)$ a.s.\
as $x\to\infty$ by the strong law of large numbers for the renewal
processes, see Theorem 5.1 on p.~57 in \cite{Gut:2009}.

\noindent (b) The limit relation that involves $N$ can be found in
the proof of Proposition 3.3 in
\cite{Alsmeyer+Iksanov+Marynych:2016}. Setting $\eta\equiv 1$
immediately gives the second limit relation.
\end{proof}

\begin{Prop}\label{aux4}
Suppose that ${\tt s}^2={\rm Var}\,\xi\in (0, \infty)$. Then
\begin{equation}\label{123}
C\bigg(\bigg(\frac{\nu(t)-{\tt m}^{-1}t}{\sqrt{2{\tt s}^2{\tt
m}^{-3}t\log\log t}}\bigg): t\geq 3 \bigg)=[-1,1]\quad\text{{\rm
a.s.}}
\end{equation}
where ${\tt m}=\me\xi<\infty$, and
\begin{equation}\label{124}
{\lim\sup}_{n\to\infty}\frac{\sup_{0\leq y\leq n}|\nu(y)-{\tt
m}^{-1}y|}{\sqrt{2n\log\log n}}={\tt s}{\tt m}^{-3/2}\quad
\text{{\rm a.s.}}
\end{equation}
\end{Prop}
\begin{Rem}
While formula \eqref{123} was known before, see, for instance,
Theorem 11.1 on p.~108 in \cite{Gut:2009}, we have not been able
to locate formula \eqref{124} in the literature. We derive both
\eqref{123} and \eqref{124} from a functional law of the iterated
logarithm. The proof of \eqref{123}, other than that mentioned on
p.~108 in \cite{Gut:2009}, is included, for it requires no extra
work in the given framework.
\end{Rem}
\begin{proof}[Proof of Proposition \ref{aux4}]
Denote by $D$ the Skorokhod space of right-continuous real-valued
functions which are defined on $[0,\infty)$ and have finite limits
from the left at each positive point. We shall need the commonly
used $J_1$-topology on $D$, see \cite{Billingsley:1968,
Whitt:2002}.

For integer $n\geq 3$, set $$X_n(t):=\frac{\nu(nt)-{\tt
m}^{-1}nt}{\sqrt{2{\tt s}^2{\tt m}^{-3}n\log\log n}},\quad t\geq
0.$$ We shall write $(X_n)$ for $(X_n(t))_{t\geq 0}$. Let $K$
denote the set of real-valued absolutely continuous functions $g$
on $[0,\infty)$ such that $g(0)=0$ and $\int_0^\infty
(g^\prime(t))^2{\rm d}t\leq 1$. The set $K$ is called the {\it
Strassen set}. It is known (see p.~44 in \cite{Vervaat:1972} or
Theorem 7.3 on p.~173 in \cite{Gut:2009}) that the sequence
$(X_n)_{n\geq 3}$ is, with probability one, relatively compact in
the $J_1$-topology, and the set of its limit points coincides with
$K$. The evaluation and the supremum functionals $h_1,h_2: D\to
\R$ defined by $h_1(x):=x(1)$ and
$h_2(x):=\sup_{t\in[0,1]}\,|x(t)|$, respectively, are continuous
in the $J_1$-topology at each $x\in K$. Hence, for $i=1,2$, by the
continuous mapping theorem $(h_i(X_n))_{n\geq 3}$ are, with
probability one, relatively compact in the $J_1$-topology, and the
sets of their limit points coincide with $h_i(K)$.

\noindent{\sc Proof of \eqref{123}}. We first show that
\eqref{123} holds with an integer argument replacing a continuous
argument. To this end, it remains to prove that $h_1(K)=[-1,1]$
which is a consequence of two facts: (I) $g(1)\in [-1,1]$ for each
$g\in K$; (II) each point of $[-1,1]$ is a possible value of
$g(1)$ for some $g\in K$.

Let $g\in K$ and $t\in (0,1]$. From
\begin{equation}\label{125}
(g(t))^2=\bigg(\int_0^t g^\prime (y){\rm d}y\bigg)^2\leq \int_0^t (g^\prime(y))^2{\rm d}y\int_0^t{\rm d}y\leq t
\end{equation}
it follows that $g(1)\in [-1,1]$. To prove (II), set $g^\pm_a(t):=\pm \min (t,a)$, for each $a\in [0,1]$. Then $g^\pm_a\in K$ and $g^\pm_a(1)=\pm a$.

Recall the notation $a(t)=\sqrt{2{\tt s}^2{\tt m}^{-3}t\log\log
t}$ for $t\geq 3$. To pass in \eqref{123} from an integer argument
to a continuous argument it is enough to check that if
$\lim_{k\to\infty}(\nu(t_k)-{\tt m}^{-1}t_k)/a(t_k) =b$ a.s.\ for
some sequence $(t_k)$ of {\it real numbers} and some $b\in\R\cup
\{\pm\infty\}$, then $\lim_{k\to\infty}(\nu(n_k)-{\tt
m}^{-1}n_k)/a(n_k)=b$ a.s.\ for some sequence $(n_k)$ of {\it
integers}. Writing
\begin{eqnarray*}
\frac{\nu([t_k])-{\tt m}^{-1}[t_k]}{a([t_k]+1)}-\frac{{\tt m}^{-1}}{a([t_k]+1)}&\leq& \frac{\nu(t_k)-{\tt m}^{-1}t_k}{a(t_k)}\leq \frac{\nu([t_k])-{\tt m}^{-1}[t_k]}{a([t_k])}\\&+&\frac{\nu([t_k]+1)-\nu([t_k])}{a([t_k])},
\end{eqnarray*}
where $[x]$ denotes the integer part of $x$,  and noting that
$\lim_{t\to\infty}(\nu(t+1)-\nu(t))/a(t)=0$ a.s.\ by Lemma
\ref{aux2}(b) and $\lim_{t\to\infty}a(t+1)/a(t)=1$ we conclude
that the implication above does indeed hold with $n_k:=[t_k]$.

\noindent{\sc Proof of \eqref{124}}. From what has been proved
above it follows that the left-hand side of \eqref{124} equals
$\sup_{g\in K}(\sup_{t\in [0,1]}\,|g(t)|)$ a.s. In view of
\eqref{125} the last expression does not exceed one. Since
$\sup_{t\in [0,1]}\,|g_1^+(t)|=1$ (recall that $g^+_1(t)=\min
(t,1)$), we infer $\sup_{g\in K}(\sup_{t\in [0,1]}\,|g(t)|)=1$
which completes the proof of \eqref{124}.
\end{proof}

\vspace{0.5cm}

\noindent   {\bf Acknowledgements}  \quad
\footnotesize The second and third authors would like to extend their sincere appreciation to the Deanship of Scientific Research at
King Saud University for funding their Research group No.  (RG-1437-020).


\normalsize


\begin{thebibliography}{30}

\bibitem{Alsmeyer+Iksanov+Marynych:2016} {\sc Alsmeyer, G., Iksanov, A. and Marynych, A.} (2017). Functional limit theorems for the number of occupied boxes in the Bernoulli sieve.
{\em Stoch. Proc. Appl.}, to appear.

\bibitem{Billingsley:1968} {\sc Billingsley, P.} (1968). {\it Convergence of probability
measures}. Wiley: New York.

\bibitem{Chow+Teicher:2003}{\sc Chow, Y.S. and Teicher, H.} (2003). {\it Probability theory: independence, interchangeability, martingales}, 3rd edition.
Springer: New York.

\bibitem{Gnedin:2004}{\sc Gnedin, A.} (2004). The Bernoulli sieve.
{\em Bernoulli}. {\bf 10}, 79--96.

\bibitem{Gnedin+Iksanov+Marynych:2010} {\sc Gnedin, A., Iksanov, A. and Marynych, A.} (2010).
Limit theorems for the number of occupied boxes in the Bernoulli
sieve. {\em Theory Stoch. Proc.} {\bf 16(32)}, 44--57.

\bibitem{Gut:2009} {\sc Gut, A.} (2009). {\it Stopped random walks. Limit theorems and
applications}, 2nd edition. Springer-Verlag: New York.

\bibitem{Vervaat:1972} {\sc Vervaat, W.} (1972). {\it Success epochs in Bernoulli trials (with applications in number theory)},
Mathematical Centre Tracts, {\bf 42}. Mathematisch Centrum: Amsterdam.

\bibitem{Whitt:2002} {\sc Whitt, W.} (2002). {\it Stochastic-process limits: an introduction to
stochastic-process limits and their application to queues.}
Springer-Verlag: New York.

\end{thebibliography}
\end{document}